\documentclass[11pt]{article}

\usepackage{amsfonts,amsmath,amssymb}
\usepackage{amsmath,amsthm}
\usepackage{latexsym}
\usepackage{color}
\usepackage{tikz}

\newtheorem{thm}{Theorem}

\newtheorem{lem}[thm]{Lemma}

\newtheorem{cor}[thm]{Corollary}

\newtheorem{ques}{Question}

\newcommand{\cp}{\,\square\,}

\newcommand{\mC}{\mathcal{C}_{\rho^{\rm o}}}
\newcommand{\mO}{\mathcal{O}}
\newcommand{\mT}{{\mathcal T}_{\rho^{\rm o}}}
\newcommand{\mop}{\rho^{\rm o}}
\newcommand{\mCtwo}{\mathcal{C}_{\rho}}
\newcommand{\mCalpha}{\mathcal{C}_{\alpha}}

\begin{document}

\title{Graphs with a unique maximum open packing}

\author{
Bo\v{s}tjan Bre\v{s}ar$^{a,b}$
\and
Kirsti Kuenzel$^{c}$
\and
Douglas F. Rall$^{d}$\\
}

\date{\today}

\maketitle

\begin{center}
$^a$ Faculty of Natural Sciences and Mathematics, University of Maribor, Slovenia\\

$^b$ Institute of Mathematics, Physics and Mechanics, Ljubljana, Slovenia\\
$^c$ Department of Mathematics, Western New England University, Springfield, MA\\
$^d$ Department of Mathematics, Furman University, Greenville, SC\\
\end{center}

\begin{abstract}
A set $S$ of vertices in a graph is an open packing if (open) neighborhoods of any two distinct vertices in $S$ are disjoint. In this paper, we consider the graphs that have a unique maximum open packing. We characterize the trees with this property by using four local operations such that any nontrivial tree with a unique maximum open packing can be obtained by a sequence of these operations starting from $P_2$. We also prove that the decision version of the open packing number is NP-complete even when restricted to graphs of girth at least $6$. Finally, we show that the recognition of the graphs with a unique maximum open packing is polynomially equivalent to the recognition of the graphs with a unique maximum independent set, and we prove that the complexity of both problems is not polynomial, unless P=NP.
\end{abstract}


\noindent
{\bf AMS Subj.\ Class.\ (2010)}: 05C70, 05C69, 05C05.

\maketitle

\section{Introduction}
A set $D$ of vertices in a graph $G$ is a {\em total dominating set of} $G$ if every vertex of $G$ has a neighbor in $D$. The minimum cardinality of a total dominating set is denoted $\gamma_t(G)$ and called the total domination number. This is one of the most studied invariants in domination theory, and has been surveyed in a monograph of Henning and Yeo~\cite{heyebook}. As pointed out in~\cite{hhs-1998}, open packings are the dual object, as seen from an integer programming standpoint, of total dominating sets. An {\em open packing} in a graph $G$ is a set of vertices whose (open) neighborhoods are pairwise disjoint. (The {\em (open) neighborhood} $N(x)$ of a vertex $x\in V(G)$ is the set $\{v\in V(G)\,:\,xv\in E(G)\}$.) By $\mop(G)$ we denote the maximum cardinality of an open packing in $G$, and call it the {\em open packing number} of $G$.

It follows from definitions of both invariants that $\gamma_t(G)\ge \mop(G)$ for an arbitrary graph $G$ with no isolated vertices. Rall studied the relationship of these two parameters and proved the following:

\begin{thm}[Rall, \cite{rall-2005}] If $T$ is a tree, then $\mop(T) = \gamma_t(T)$.
\end{thm}

\noindent He then used this result to prove a formula for the total domination number of a direct product of a nontrivial tree $T$ with an arbitrary graph $H$, which reads $\gamma_t(T\times H) = \gamma_t(T)\gamma_t(H)$. For other relations between open packing number and total domination number see~\cite{gzsgk-2018,gh-2018,hrs-2013}.

Graphs which possess a unique minimum total dominating set were studied by Haynes and Henning in~\cite{hahe-2002}, who characterized the trees having this property. In spite of the fact that $\mop(T) = \gamma_t(T)$ when $T$ is a tree, there is seemingly no connection between the trees with a unique maximum open packing and the trees with a unique minimum total dominating set. For instance, the path $P_4$ is a tree with a unique minimum total dominating set, while there are several pairs of vertices in $P_4$ that form a maximum open packing. On the other hand, the path $P_6$ is a small example of a tree that has a unique maximum open packing but has more than one minimum total dominating set.

Trees with a unique maximum independent set were studied by Hopkins and Staton~\cite{hs-1985} and by Gunther et al.~\cite{ghr-1993}. Recently, Jaume and Molina~\cite{jm-2018} provided a characterization of such trees from an algebraic point of view, which can be used for efficient recognition of trees with a unique maximum independent set; see Section~\ref{sec:complexity}. Trees that have a unique maximum $2$-packing (where $2$-packing is a set of vertices having pairwise disjoint closed neighborhoods) were also characterized recently~\cite{bope-2019+}. In this paper, we begin the study of graphs which have a unique maximum open packing, and we denote the set of such graphs by $\mC$. In particular, we show that given a graph $G \in \mC$, we can perform one of four operations on $G$ to create a larger graph $G'$ which also belongs to $\mC$. We then characterize all trees in $\mC$ by showing that any such (nontrivial) tree can be obtained with just these four operations from the two-vertex tree.

We also study complexity issues in identifying graphs in $\mC$. In doing so, we are able to show that the open packing number is related to two well-known graph invariants. Recall that the maximum cardinality of an independent set of vertices in a graph $G$ is the {\em independence number} of $G$ and is denoted by $\alpha(G)$. By $\alpha'(G)$ we denote the maximum cardinality of a matching in $G$ (where a {\em matching} is a set of independent edges). The {\em subdivision of $G$} is the graph $S(G)$ obtained by subdividing every edge of $G$ by exactly one vertex. We will prove that $\mop(S(G)) = \alpha(G) + \alpha'(G)$.

The remainder of this paper is organized as follows. In Section~\ref{sec:meat}, we characterize all trees that have a unique maximum open packing. We do so by showing that every nontrivial tree with a unique maximum open packing can be obtained by starting with a single edge and performing a series of operations. In Section~\ref{sec:complexity}, we show that $\mop(S(G)) = \alpha(G) + \alpha'(G)$ and we discuss the complexity of identifying whether an arbitrary graph, or an arbitrary tree, contains a unique maximum open packing. In particular, we prove that the problems of recognizing the graphs with a unique maximum open packing, a unique maximum 2-packing, and a unique maximum independent set, respectively, are all polynomially equivalent problems (i.e., the complexity of recognizing the graphs in these classes is essentially the same). Finally, we prove that these recognition problems are not polynomial, unless P=NP.

\section{Trees with unique maximum open packings}
\label{sec:meat}
For any $G \in \mC$, we let $U(G)$ represent the unique maximum open packing in $G$ and we refer to $U(G)$ as the {\em $\mop(G)$-set}. In this section, we will characterize all trees in $\mC$. To that end, we begin with the following preliminary results that apply to all graphs (not just trees) in $\mC$. Recall that $x \in V(G)$ is a {\em leaf} in $G$ if $\deg(x) = 1$ (where the degree, $\deg(x)$, of $x$ is the cardinality of its neighborhood), and we refer to the neighbor of $x$ as its {\em support vertex}. Additionally, we say that $v$ is a {\em strong support vertex} if $v$ is adjacent to more than one leaf.
\vskip5mm

\begin{lem}\label{lem:leaf}
If $G \in \mC$, then every leaf of $G$ belongs to $U(G)$.
\end{lem}
\begin{proof}
Let $x$ be a leaf and $y$ be its support vertex. We may assume there exists a vertex $z \in U(G)$ that is adjacent to $y$ for otherwise $U(G) \cup \{x\}$ is a larger open packing of $G$.
If $z \not= x$, then $U' = (U(G) -  \{z\})\cup \{x\}$ is an open packing, contradicting the uniqueness of $U(G)$.
\end{proof}

\begin{cor} \label{cor:no11}
If $G \in \mC$, then $G$ has no strong support vertices.
\end{cor}

\begin{proof}
Let $u$ and $v$ be leaves in $G$.  By Lemma~\ref{lem:leaf}, $U(G)$ contains both $u$ and $v$.  Since $U(G)$ is an open packing, $N(u) \cap N(v)=\emptyset$.  We infer that $u$ and $v$ are not adjacent to the same support vertex.
\end{proof}

\begin{lem}\label{lem:no21}
Let $G$ be a graph.  If $\ell$ is a leaf in $G$ with support vertex $x$ and $z$ is a leaf with support vertex $y$ such that $N(y) = \{z, x\}$, then $G \not\in \mC$.
\end{lem}
\begin{proof}
Suppose $G \in \mC$. By Lemma~\ref{lem:leaf}, $z, \ell \in U(G)$, which implies $x \not\in U(G)$ and $y \not\in U(G)$. Hence, $U' = (U(G) -  \{\ell\}) \cup \{y\}$ is another open packing of size $\mop(G)$, which contradicts the fact that $U(G)$ is the unique maximum open packing of $G$.
\end{proof}

\begin{lem}\label{lem:no22}
Let $G$ be a graph. If $\ell_1$ and $\ell_2$ are leaves of $G$ with support vertices $y_1$ and $y_2$, respectively, $N(y_1) = \{\ell_1, x\}$ and $N(y_2) = \{\ell_2,x\}$, then $G \not\in \mC$.
\end{lem}

\begin{proof}
Suppose $G \in \mC$. By Lemma~\ref{lem:leaf}, $\{\ell_1, \ell_2\} \subseteq U(G)$, which implies that $x \not\in U(G)$. The open neighborhood of $x$ intersects $U(G)$, for otherwise $U(G) \cup \{y_1\}$ is a larger open packing.  If $z \in U(G) \cap N(x)$ and $z \not \in \{y_1,y_2\}$, then $(U(G)  -  \{z\}) \cup \{y_1\}$ is another maximum open packing, which is a contradiction. Thus, assume without loss of generality that $U(G) \cap N(x)=\{y_1\}$.  But now $(U(G)  -  \{y_1\}) \cup \{y_2\}$ is a different maximum open packing, which again contradicts the assumption that $G \in \mC$.
\end{proof}

%%%%%%%%%%%%%%%%%%%%
%%%%%%%%%%%%%%%%%%%%%%%%%%%%%%%%OPERATIONS%%%%%%%%%%%
%%%%%%%%%%%%%%%%%%%%%

We now define four operations on a graph $G$ that belongs to $\mC$.  These operations will be used in the remainder of the paper to build
additional graphs in $\mC$.  By \emph{appending} a path $P_n$ to a vertex $x \in V(G)$ we mean make $x$ adjacent to a leaf of a (new) path of order $n$. An edge $e\in E(G)$ is called a {\em cut edge} of $G$ if $e$ does not lie on any cycle, and hence the endvertices of $e$ belong to two distinct connected components in $G-e$.

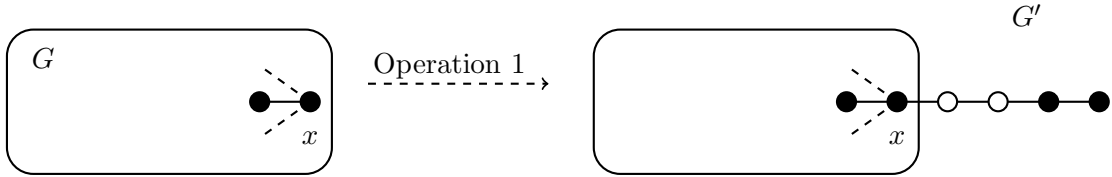
\begin{figure}[!ht]
\centering
\begin{tikzpicture}[scale=1.2, style=thick]

\coordinate(general)(0,0);
\begin{scope}[shift=(general), scale=0.8]
\def\vr{3pt / 0.8}
\draw[rounded corners=10pt] (0,0) rectangle (4.5,2);

\coordinate(x) at (4.2,1);
\coordinate(u) at (3.5,1);

\draw(4.2,0.5)node{$x$};
\draw(0.5,1.6)node{$G$};

\draw(x)--(u);
\draw[dashed](4.2, 1)--(3.5,1.5);
\draw[dashed](4.2, 1)--(3.5,0.5);
\draw(x)[fill=black] circle(\vr);
\draw(u)[fill=black] circle(\vr);
\end{scope}

%%%%%%%%
\draw[dashed,->](4, 1)--(6,1);
\draw(4.9,1.2)node{Operation 1};

%%%%%%%%

\coordinate(small) at (6.5,0);
\begin{scope}[shift=(small), scale=0.8]
\def\vr{3pt / 0.8}
\draw[rounded corners=10pt] (0,0) rectangle (4.5,2);

\coordinate(x) at (4.2,1);
\coordinate(u) at (3.5,1);
\coordinate(a) at (4.9,1);
\coordinate(b) at (5.6,1);
\coordinate(c) at (6.3,1);
\coordinate(d) at (7,1);

\draw(4.2,0.5)node{$x$};
\draw(6,2.2)node{$G'$};

\draw (u)--(x);
\draw (x)--(a);
\draw (a)--(b);
\draw (b)--(c);
\draw (c)--(d);
\draw[dashed](4.2, 1)--(3.5,1.5);
\draw[dashed](4.2, 1)--(3.5,0.5);
\draw(x)[fill=black] circle(\vr);
\draw(u)[fill=black] circle(\vr);
\draw(a)[fill=white] circle(\vr);
\draw(b)[fill=white] circle(\vr);
\draw(c)[fill=black] circle(\vr);
\draw(d)[fill=black] circle(\vr);

\end{scope}

\end{tikzpicture}

\caption{Operation 1; $G\in\mC$, $x$ and one of its neighbors are in $U(G)$;  vertices of the unique maximum open packing are black.}
\label{fig:not-3-2}
\end{figure}

%%%%%%%%%%%%%%%%%%%%%%%%%%%%%%%%%%%
%%%%%%%%%%%%%%%%%%%%%%%%%%%%%%%%%%%

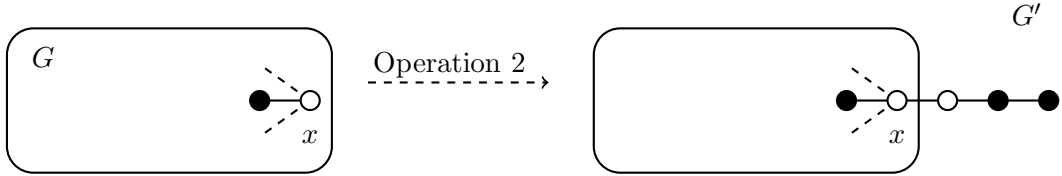
\begin{figure}[!ht]
\centering
\begin{tikzpicture}[scale=1.2, style=thick]

\coordinate(general)(0,0);
\begin{scope}[shift=(general), scale=0.8]
\def\vr{3pt / 0.8}
\draw[rounded corners=10pt] (0,0) rectangle (4.5,2);

\coordinate(x) at (4.2,1);
\coordinate(u) at (3.5,1);

\draw(4.2,0.5)node{$x$};
\draw(0.5,1.6)node{$G$};

\draw(x)--(u);
\draw[dashed](4.2, 1)--(3.5,1.5);
\draw[dashed](4.2, 1)--(3.5,0.5);
\draw(x)[fill=white] circle(\vr);
\draw(u)[fill=black] circle(\vr);
\end{scope}

%%%%%%%%
\draw[dashed,->](4, 1)--(6,1);
\draw(4.9,1.2)node{Operation 2};

%%%%%%%%

\coordinate(small) at (6.5,0);
\begin{scope}[shift=(small), scale=0.8]
\def\vr{3pt / 0.8}
\draw[rounded corners=10pt] (0,0) rectangle (4.5,2);

\coordinate(x) at (4.2,1);
\coordinate(u) at (3.5,1);
\coordinate(a) at (4.9,1);
\coordinate(b) at (5.6,1);
\coordinate(c) at (6.3,1);

\draw(4.2,0.5)node{$x$};
\draw(6,2.2)node{$G'$};

\draw (u)--(x);
\draw (x)--(a);
\draw (a)--(b);
\draw (b)--(c);
\draw[dashed](4.2, 1)--(3.5,1.5);
\draw[dashed](4.2, 1)--(3.5,0.5);
\draw(x)[fill=white] circle(\vr);
\draw(u)[fill=black] circle(\vr);
\draw(a)[fill=white] circle(\vr);
\draw(b)[fill=black] circle(\vr);
\draw(c)[fill=black] circle(\vr);

\end{scope}

\end{tikzpicture}

\caption{Operation 2; $G\in\mC$, $x\notin U(G)$, but one of its neighbors belongs to $U(G)$. }
\label{fig:not-3-2}
\end{figure}

%%%%

\begin{itemize}
\item {\underline{Operation 1:}} Let $x \in V(G)$ such that $x$ and some neighbor of $x$ both belong to $U(G)$.  Append a $P_4$ to $x$. The resulting graph $G'$ is said to be obtained by Operation $1$ at $x$.
\item {\underline{Operation 2:}} Let $x \in V(G)$ such that $x$ does not belong to $U(G)$, but some neighbor of $x$ is in $U(G)$.  Append a $P_3$ to $x$. The resulting graph $G'$ is said to be obtained by Operation $2$ at $x$.
\item {\underline{Operation 3:}} Let $x \in V(G)$ such that $x$ belongs to $U(G)$, some neighbor of $x$ belongs to $U(G)$, and there exists exactly one open packing of $G$ of size $\mop(G) - 1$ that does not contain $x$. (Equivalently, the third condition states that $G-x \in \mC$.) Append a $P_3$ to $x$ and then append another $P_3$ to $x$. The resulting graph $G'$ is said to be obtained by Operation $3$ at $x$.
\item{\underline{Operation 4:}} Let $xy$ be a cut edge of $G$ such that neither $x$ nor $y$ belongs to $U(G)$, but both $x$ and $y$
    have a neighbor that belongs to $U(G)$. Subdivide the edge $xy$ and add a leaf adjacent to the new vertex of the subdivided edge. The resulting graph $G'$ is said to be obtained by Operation $4$ at $xy$.
\end{itemize}

%%%%%
%%%%

\begin{figure}[!ht]
\centering
\begin{tikzpicture}[scale=1.2, style=thick]

\coordinate(general)(0,0);
\begin{scope}[shift=(general), scale=0.8]
\def\vr{3pt / 0.8}
\draw[rounded corners=10pt] (0,0) rectangle (4.5,2);
\draw[densely dotted,rounded corners=10pt] (0,0) rectangle (3.9,2);

\coordinate(x) at (4.2,1);
\coordinate(u) at (3.5,1);

\draw(4.2,0.5)node{$x$};
\draw(0.5,1.6)node{$G$};

\draw(x)--(u);
\draw[dashed](4.2, 1)--(3.5,1.5);
\draw[dashed](4.2, 1)--(3.5,0.5);
\draw(x)[fill=black] circle(\vr);
\draw(u)[fill=black] circle(\vr);
\end{scope}

%%%%%%%%
\draw[dashed,->](4, 1)--(6,1);
\draw(4.9,1.2)node{Operation 3};

%%%%%%%%

\coordinate(small) at (6.5,0);
\begin{scope}[shift=(small), scale=0.8]
\def\vr{3pt / 0.8}
\draw[rounded corners=10pt] (0,0) rectangle (4.5,2);
\draw[densely dotted,rounded corners=10pt] (0,0) rectangle (3.9,2);

\coordinate(x) at (4.2,1);
\coordinate(u) at (3.5,1);
\coordinate(a) at (4.8,0.8);
\coordinate(b) at (5.4,0.6);
\coordinate(c) at (6,0.4);
\coordinate(a') at (4.8,1.2);
\coordinate(b') at (5.4,1.4);
\coordinate(c') at (6,1.6);

\draw(4.2,0.5)node{$x$};
\draw(5.8,2.3)node{$G'$};

\draw (u)--(x);
\draw (x)--(a);
\draw (a)--(b);
\draw (b)--(c);
\draw[dashed](4.2, 1)--(3.5,1.5);
\draw[dashed](4.2, 1)--(3.5,0.5);

\draw(u)[fill=black] circle(\vr);
\draw(a)[fill=white] circle(\vr);
\draw(b)[fill=black] circle(\vr);
\draw(c)[fill=black] circle(\vr);
\draw (x)--(a');
\draw (a')--(b');
\draw (b')--(c');
\draw(a')[fill=white] circle(\vr);
\draw(b')[fill=black] circle(\vr);
\draw(c')[fill=black] circle(\vr);
\draw(x)[fill=white] circle(\vr);

\end{scope}

\end{tikzpicture}

\caption{Operation 3; both $G$ and $G-x$ belong to $\mC$.}
\label{fig:not-3-2}
\end{figure}
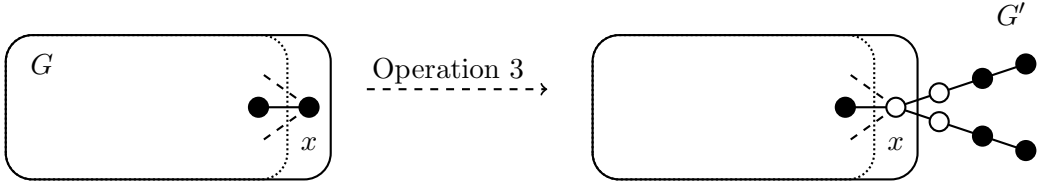

%%%%
%%%%

\begin{figure}[!ht]
\centering
\begin{tikzpicture}[scale=1.2, style=thick]

\coordinate(general)(0,0);
\begin{scope}[shift=(general), scale=0.8]
\def\vr{3pt / 0.8}

\draw[densely dotted,rounded corners=10pt] (0,0) rectangle (2.3,1.9);
\draw[densely dotted,rounded corners=10pt] (2.5,0) rectangle (4.6,1.9);

\coordinate(x) at (2,1);
\coordinate(y) at (2.8,1);
\coordinate(u) at (1.2,1);
\coordinate(v) at (3.6,1);

\draw(2,0.6)node{$x$};
\draw(2.8,0.6)node{$y$};
\draw(0.5,2.4)node{$G$};

\draw(x)--(y);
\draw(x)--(u);
\draw(y)--(v);
\draw[dashed](2.8, 1)--(3.4,1.5);
\draw[dashed](2.8, 1)--(3.4,0.5);
\draw[dashed](2, 1)--(1.4,1.5);
\draw[dashed](2, 1)--(1.4,0.5);
\draw(x)[fill=white] circle(\vr);
\draw(y)[fill=white] circle(\vr);
\draw(u)[fill=black] circle(\vr);
\draw(v)[fill=black] circle(\vr);
\end{scope}

%%%%%%%%
\draw[dashed,->](4, 1)--(6,1);
\draw(4.9,1.2)node{Operation 4};

%%%%%%%%

\coordinate(small) at (6.5,0);
\begin{scope}[shift=(small), scale=0.8]
\def\vr{3pt / 0.8}

\draw[densely dotted,rounded corners=10pt] (0,0) rectangle (2.3,1.9);
\draw[densely dotted,rounded corners=10pt] (3.2,0) rectangle (5.5,1.9);

\coordinate(x) at (2,1);
\coordinate(y) at (3.6,1);
\coordinate(z) at (2.8,1);
\coordinate(w) at (2.8,1.6);
\coordinate(u) at (1.2,1);
\coordinate(v) at (4.6,1);

\draw(2,0.6)node{$x$};
\draw(3.6,0.6)node{$y$};
\draw(5.3,2.3)node{$G'$};

\draw(x)--(z);
\draw(x)--(u);
\draw(y)--(z);
\draw(y)--(v);
\draw(z)--(w);
\draw[dashed](3.6, 1)--(4.2,1.5);
\draw[dashed](3.6, 1)--(4.2,0.5);
\draw[dashed](2, 1)--(1.4,1.5);
\draw[dashed](2, 1)--(1.4,0.5);
\draw(x)[fill=white] circle(\vr);
\draw(y)[fill=white] circle(\vr);
\draw(u)[fill=black] circle(\vr);
\draw(v)[fill=black] circle(\vr);
\draw(z)[fill=white] circle(\vr);
\draw(w)[fill=black] circle(\vr);

\end{scope}

\end{tikzpicture}

\caption{Operation 4; $G\in\mC$, and $xy$ is a cut edge in $G$.}
\label{fig:not-3-2}
\end{figure}
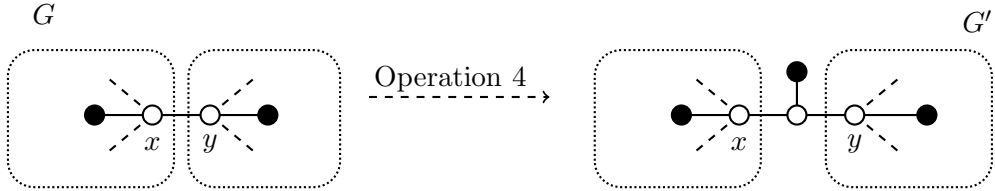

%%%%%%%%%%%%%%%%%%%%%%%%%%%%%%%%%%%
%%%%%%%%%%%%%%%%%%%%%%%%%%%%%%%%%%%

We first show that if $G \in \mC$ and $G'$ is obtained from $G$ by performing one of the above operations, then  $G' \in\mC$.

\begin{lem}\label{lem:op1}
Let $G \in \mC$ and $x \in U(G)$ such that there exists $y \in N(x) \cap U(G)$. If $G'$ is the graph obtained from $G$ by Operation $1$ at $x$, then $G' \in \mC$.
\end{lem}
\begin{proof}
Label the vertices of the appended $P_4$ $a,b,c$ and $d$ so that $xa \in E(G')$. We let $\mop(G) = k$. It follows that $\mop(G') \ge k+2$, since $U(G)\cup\{c,d\}$ is an open packing of $G'$. Let $S$ be a maximum open packing of $G'$. Note that $|S\cap\{a,b,c,d\}|\le 2$, which implies that $S\cap V(G)$ is an open packing of $G$ of size at least $k$. Hence $|S\cap\{a,b,c,d\}|=2$ and $S\cap V(G)=U(G)$. Hence $x,y\in S$, which implies that $S\cap\{a,b,c,d\}=\{c,d\}$, and $S$ is unique. Therefore, $G'\in \mC$.
\end{proof}

 \begin{lem}\label{lem:op2}
 Let $G \in \mC$ and $x\not\in U(G)$ such that $x$ has a neighbor $y\in U(G)$. If $G'$ is the graph obtained from $G$ by Operation $2$ at $x$, then $G' \in \mC$.
 \end{lem}

 \begin{proof}
 Label the vertices of the appended $P_3$ $a,b,$ and $c$ so that $xa \in E(G')$. We let $\mop(G) = k$. It follows that $\mop(G') \ge k+2$ for $U(G) \cup \{b, c\}$ is an open packing of $G'$. If $M$ is a $\mop(G')$-set, then $|M \cap V(G)| \le k$, which implies that $\mop(G') \le k+2$. Therefore, $\mop(G') = k+2$. To see that $G'\in \mC$, suppose there exists another $\mop(G')$-set $T$ different from $U(G) \cup \{b,c\}$. We may assume $|T\cap \{a, b, c \}|\le 1$ for otherwise $T\cap V(G)$ is a $\mop(G)$-set different from $U(G)$. However, this implies $|T\cap V(G)| \ge k+1$, contradicting the assumption that $\mop(G)=k$. Therefore, no such $T$ exists and $G' \in \mC$.
 \end{proof}

 \begin{lem}\label{lem:op3}
 Let $G\in \mC$ and $x\in U(G)$ such that $N(x) \cap U(G) \ne \emptyset$. Moreover, suppose $G$  has exactly one open packing of cardinality $\mop(G) - 1$ that does not contain $x$. If $G'$ is the graph obtained from $G$ by Operation $3$ at $x$, then $G' \in \mC$.
 \end{lem}

 \begin{proof}
 Label the vertices of the two appended paths of order $3$ as $a,b,$ and $c$, and $a',b'$ and $c'$, respectively, where $xa \in E(G')$ and $xa'\in E(G')$. We let $\{y \} = U(G) \cap N(x)$ and we assume that $\mop(G) = k$. We claim that $\mop(G') = k+3$. Certainly $(U(G) -  \{x\})\cup \{b, b', c, c'\}$ is an open packing of $G'$, and hence $\mop(G') \ge k+3$. Suppose that $\mop(G') \ge k+4$. Let $T$ be an open packing in $G'$ of cardinality $k+4$. If $|T \cap \{a, a', b, b', c, c'\}| \le 3$, then $|T\cap V(G)| \ge k+1$, which contradicts our assumption that $\mop(G) = k$. Thus, $|T \cap \{a, a', b, b', c, c'\}| = 4$. If $a \in T$, then $T\cap V(G)$ is another $\mop(G)$-set different from $U(G)$ since $y \not\in T \cap V(G)$. Therefore, $\{b, b', c, c'\} \subset T$ and $T\cap V(G)$ is a $\mop(G)$-set different from $U(G)$ as $x \not\in T\cap V(G)$. Therefore, no such set $T$ exists and $\mop(G') = k+3$.

 Let $M$ be a maximum open packing of $G'$. Notice that $|M \cap \{a,a',b, b', c, c'\}| \ge 3$ for otherwise $|M \cap V(G)| \ge k+1$, which is a contradiction. Thus, $\{b, b'\} \cap M \ne \emptyset$. This implies that $x \notin M$. Now, if $a$ or $a'$ belongs to $M$, then $y\notin M$, which implies that $M\cap V(G)$ is an open packing of $G$ of size at least $k-1$ such that $M\cap \{x,y\}=\emptyset$. This contradicts the assumptions on $G$ since the open packing of $G$ of size $\mop(G)$ is unique and since there is only one open packing of $G$ of size $\mop(G)-1$ that does not contain $x$. Thus, neither $a$ nor $a'$ is in $M$. This implies that $\{b, b', c, c'\} \subset M$, for otherwise $M$ is not maximum. Therefore, $|M \cap V(G)| = k-1$ and hence  $M\cap V(G)$ is unique, which implies that $M$ is unique. Thus, $G' \in \mC$.
 \end{proof}

 \begin{lem}\label{lem:op4}
 Let $G\in \mC$ and let $xy$ be a cut edge in $G$ such that neither $x$ nor $y$ belongs to $U(G)$ yet $N(x) \cap U(G) \ne \emptyset$ and $N(y) \cap U(G) \ne \emptyset$.
 If $G'$ is the graph obtained from $G$ by Operation $4$ at $xy$, then $G' \in \mC$.
 \end{lem}

 \begin{proof}
 We let $b$ represent the vertex of the subdivided edge and let $a$ be the leaf adjacent to $b$. We shall assume that $\mop(G) = k$. We also let $\{w\} = N(x)\cap U(G)$ and $\{z\} = N(y) \cap U(G)$. Note that $\mop(G') \ge k+1$ since $U(G) \cup \{a\}$ is an open packing of $G'$. Suppose there exists an open packing $M$ of $G'$ of size $k+2$. If $|M \cap \{x, y, a, b\}| = 2$, then clearly $b\in M$, but then $M\cap V(G)$ is a different open packing from $U(G)$ of size $k$, which is a contradiction. Therefore, $|M \cap \{x, y, a, b\}| \le 1$.
This implies that $|M\cap (V(G')-\{x, y, a, b\})|\ge k+1$. Note that $M\cap (V(G')-\{x, y, a, b\})$ is an open packing also in $G$, which contradicts $\mop(G)=k$.
 Therefore, no such $M$ exists and we conclude that $\mop(G') = k+1$. Clearly, $T = U(G) \cup \{a\}$ is an open packing of size $\mop(G')$.

To see that $G' \in \mC$, suppose $S$ is a $\mop(G')$-set of cardinality $k+1$. Let $L$ be the component of $G-y$ containing $x$ and let $R$ be the component of $G - x$ containing $y$. Note that $U(G) \cap V(L)$ is the unique maximum open packing of $L$ that does not contain $x$, and $U(G) \cap V(R)$ is the unique maximum open packing of $R$ that does not contain $y$.
Let $|V(L) \cap U(G)| = \ell$ and $|V(R) \cap U(G)| = r$.
We claim that $|S \cap V(L)| \le \ell$. Otherwise, $(S\cap V(L))\cup (U(G)\cap (V(R)-\{z\}))$ is an open packing of $G$ of cardinality at least $k$ different from $U(G)$, a contradiction. In a similar way we can prove that $|S\cap V(R)|\le r$.

Suppose $x\in S$. Then $\{a,y\}\cap S=\emptyset$. We have two cases with respect to $b$ belonging to $S$. If $b\notin S$, then $|S\cap V(R)|\ge r+1$, which contradicts the above observation. On the other hand, if $b\in S$, then $S-\{b\}$ is an open packing of $G$ of size $k$ different from $U(G)$, which contradicts the assumption that $G\in \mC$. This implies that $x\notin S$. In a similar way one proves that $y\notin S$.

Finally, we claim that $b\notin S$. Suppose $b\in S$. From this we derive that $N(y)\cap S=\{b\}=N(x)\cap S$. Hence, $S\cap V(L)$ is an open packing of $L$ that does not contain $x$. Since $S\cap V(L)$ also does not contain $w$, we infer that $|S\cap V(L)|<\ell$, because there is just one open packing of size $\ell$ in $L$ that does not contain $x$, and it contains $w$. Similarly, $|S\cap V(R)|<r$. But now, we infer that $k+1=|S|\le (\ell-1)+(r-1)+2\le k$, a contradiction, which yields that $b\notin S$. Since $S$ is a maximum open packing of $G'$, we infer that $|V(L) \cap S| = \ell$, $|V(R) \cap S| = r$ and $a\in S$. This readily implies that $S=T$.

Therefore, $U(G) \cup \{a\}$ is the only $\mop(G')$-set, and hence $G'\in \mC$.
\end{proof}

%%%%%%%%%%%%%%%%%%%%%%%%%%%%
%%%%%%%%%%%%%%%%%% REVERSED DIRECTION
%%%%%%%%%%%%%%%%%%%%%%%%%%%%%%%

From this point on we concentrate on the class of trees in $\mC$. For this purpose we construct the class $\mO$ as follows.  A tree $T$ is in $\mO$ if and only if $T=P_1$, $T=P_2$ or $T$ can be constructed from $P_2$ by a finite sequence of steps, each one of which is Operation $1$, Operation $2$, Operation $3$, or Operation $4$.

Let $\mT$ be the class of all trees that have a unique maximum open packing. Lemmas~\ref{lem:op1},~\ref{lem:op2},~\ref{lem:op3} and~\ref{lem:op4} show that $\mO\subseteq\mT$. To prove $\mT\subseteq \mO$ we will use induction on the order of a tree. We will show that $T'\in \mT$ implies that $T'$ can be obtained from a smaller tree, which is in $\mT$, by using one of the four operations. The induction hypothesis implies that this smaller tree is also in $\mO$, which finally implies that $T'\in\mO$.

The basis of induction with trees $P_1$ and $P_2$ is clear by definition of $\mO$.
Let $T'$ be a tree in $\mT$ with $|V(T')|>2$.

If $\Delta(T')=2$, then $T'$ is a path, and the following observation is not hard to see (its proof follows by using Lemma~\ref{lem:claim0}).

\begin{lem}\label{lem:paths} If $T'=P_n$, $n\ge 2$, then $T'\in \mT$ if and only if $n\equiv 2\pmod 4$.
\end{lem}

Hence, the case when $T'$ is a path is settled ($T'\in \mT$ implies $T'\in \mO$). Now, let $\Delta(T')\ge 3$. We define a tree $L(T')$ as follows. Vertices of $L(T')$ are vertices of $T'$ of degree at least $3$, and two vertices $p,q$ of $L(T')$ are adjacent if the internal vertices in the unique path from $p$ to $q$ in $T'$ are of degree $2$. In each of the subsequent lemmas we consider vertex $x\in V(T')$, which is a leaf in $L(T')$. In other words, $\deg_{T'}(x)\ge 3$, and all components of $T'-x$, with only one possible exception, are paths.

\begin{lem}\label{lem:claim0} Let $T' \in \mT$ and let $x$ be a leaf in $L(T')$. Suppose there is a path $x=u_0,u_1,\ldots,u_k$ in $T'$, where $\deg_{T'}(u_i)=2$ for $i\in [k-1]$, $\deg_{T'}(u_k)=1$, and $k\ge 4$. The tree $T=T'-\{u_{k-3},\ldots,u_k\}$ belongs to $\mT$. In addition, $T'$ is obtainable from $T$ by Operation 1 at $u_{k-4}$.
\end{lem}

\begin{proof}
Let $U(T')$ be the unique $\mop(T')$-set. By Lemma~\ref{lem:leaf}, $u_k \in U(T')$. If $u_{k-3} \in U(T')$, then $U''=(U(T') - \{u_{k-3}\}) \cup \{u_{k-1}\}$ is $\mop(T')$-set different from $U(T')$, which contradicts the uniqueness of $U(T')$. Hence, $\{u_{k-1},u_k\} \subset U(T')$. Let $M= U(T')\cap V(T)$. If $u_{k-4}$ is not totally dominated by $M$, then $U''= (U(T') -  \{u_{k-1}\})\cup \{u_{k-3}\}$ is another $\mop(T')$-set different from $U(T')$, another contradiction. Therefore, $N_T(u_{k-4}) \cap M \ne \emptyset$. Denote by $z$ the vertex in $N_T(u_{k-4}) \cap M$. If $u_{k-4} \not\in U(T')$, then $U'''=(U(T') - \{u_{k}\})\cup \{u_{k-2}\}$ is another $\mop(T')$-set. Therefore, $u_{k-4} \in M$. Note that $M$ is a $\mop(T)$-set, and that $M$ contains $u_{k-4}$ and its neighbor $z$.

Finally, to see that $T \in \mT$, suppose that $S$ is a $\mop(T)$-set different from $M$. Then $S\cup \{u_{k-1},u_k\}$ and $M\cup \{u_{k-1},u_k\}$ are two different $\mop(T')$-sets, which is a contradiction. Therefore, $T\in \mT$ and note that $T'$ is obtained from $T$ by Operation 1 at $u_{k-4}$.
\end{proof}

Let $x$ be a leaf in $L(T')$. By using Lemma~\ref{lem:claim0} (multiple times if necessary) we may assume that $k\le 3$ for any path $x=u_0,u_1,\ldots,u_k$, where $\deg_{T'}(u_i)=2$ for $i\in [k-1]$ and $\deg_{T'}(u_k)=1$. In addition, by Corollary~\ref{cor:no11} and Lemmas~\ref{lem:no21} and~\ref{lem:no22}, we infer that at most one such path is of length $k \in \{1,2\}$, and all other such paths are of length $k=3$ (since $\deg_{T'}(x)\ge 3$ there is at least one such path of length $3$).

\begin{lem}\label{lem:claim33} Let $T' \in \mT$ and let $x$ be a leaf in $L(T')$. Suppose there are two paths $x,a,b,c$ and $x,a',b',c'$, where $\deg_{T'}(a)=\deg_{T'}(a')=\deg_{T'}(b)=\deg_{T'}(b')=2$, and $\deg_{T'}(c)=\deg_{T'}(c')=1$. If $T_1 = T'- \{a, b, c\}$ and $T = T' - \{a, b, c, a', b', c'\}$, then at least one of the trees $T,T_1$ is in $\mT$. In addition, $T'$ is obtainable from $T$ by Operation 3 at $x$, or $T'$ is obtainable from $T_1$ by Operation 2 at $x$.
\end{lem}

\begin{proof}   By Lemma~\ref{lem:leaf}, $\{c, c'\}\subset U(T')$. If $a\not\in N(U(T'))$, then $U(T')$ is not maximum because $U(T') \cup \{b\}$ is also an open packing of $T'$. Therefore, either $x \in U(T')$ or $b \in U(T')$. If $x \in U(T')$, then $(U(T') -  \{x\})\cup \{b, b'\}$ is an open packing of $T'$ of cardinality larger than $\mop(T')$, a contradiction. Thus, $\{b, b'\} \subset U(T')$.

Let $T_1 = T' - \{a, b, c\}$.  Let $\mop(T') = k$ and note that $U(T') - \{b,c\}$ is an open packing of $T_1$.  Hence $\mop(T_1) \ge k-2$. Suppose that $\mop(T_1) \ge k-1$ and let $X$ be a $\mop(T_1)$-set. It follows that $x \in X$, for otherwise $X \cup \{b,c\}$ is an open packing of $T'$ of cardinality $k+1$. However, $X \cup \{c\}$ is then a different maximum open packing of $T'$, which is a contradiction. Therefore, $\mop(T_1) = k-2$.

Next note that if $x$ is not in $N(U(T'))$, then $(U(T') -  \{c'\}) \cup \{a'\}$ is a different $\mop(T')$-set. Therefore, $x$ is totally dominated by $U(T')$ and there exists $y \in N_T(x) \cap U(T')$. Note, that if there exist two different $\mop(T_1)$-sets that do not contain $x$, then there exist two different $\mop(T')$-sets, which is a contradiction. This implies that $W_1=U(T')-\{b,c\}$ is the unique maximum open packing of $T_1$ that does not contain $x$. Now, if $T_1 \in \mT$, then $T'$ is obtained from $T_1$ using Operation $2$ at $x$.

Thus, we may assume $T_1 \not\in \mT$ and all $\mop(T_1)$-sets except $W_1$ contain $x$. Let $P$ be a $\mop(T_1)$-set that contains $x$. We know that $P$ contains precisely one of $a'$ or $c'$. Therefore, $\mop(T) \ge k-3$ since $P - \{a', c'\}$ is an open packing of $T$. On the other hand, if there exists an open packing $W$ of $T$, of cardinality $k-2$, then $W \cup \{c, c'\}$ is a different $\mop(T')$-set. Therefore $\mop(T) = k-3$. Let $R = V(T) \cap P$ and $R' = (R -  \{x\}) \cup \{b, b', c, c'\}$. We infer that $R' = U(T')$, for otherwise there would be two different $\mop(T')$-sets. This implies that there is exactly one $\mop(T)$-set that contains $x$. Moreover, every $\mop(T)$-set $S$ contains $x$, for otherwise $S\cup\{b,b',c,c'\}$ would be an open packing of $T'$ of size $k+1$. This in turn implies that $T\in\mT$ and the unique $\mop(T)$-set is $R$. In addition, $R  -  \{x\}$ is a unique open packing of $T$ that does not contain $x$ and is of size $\mop(T)-1$ (because $(R - \{x\})\cup \{b, b', c, c'\}$ is $U(T')$, which contains a neighbor of $x$).  So we can obtain $T'$ from $T$ by using Operation $3$ at $x$.
\end{proof}

Applying Lemma~\ref{lem:claim33} repeatedly we may assume that $k=3$ for at most one path $x=u_0,u_1,\ldots,u_k$, where $\deg_{T'}(u_i)=2$ for $i\in [k-1]$ and $\deg_{T'}(u_k)=1$. In addition, by Corollary~\ref{cor:no11} and Lemmas~\ref{lem:no21} and~\ref{lem:no22}, we may assume that there is one such path with $k=3$, for otherwise $\deg_{T'}(x)<3$ and $x$ is therefore not a leaf of $L(T')$. For the same reason (i.e., $\deg_{T'}(x)$ must be at least $3$), there is another path, which is either $x,e$ or $x,d,e$, where $\deg_{T'}(d)=2$ and $\deg_{T'}(e)=1$.
In the next lemma we deal with the former.

\begin{lem}\label{lem:claim31} Let $T' \in \mT$ and let $x$ be a leaf in $L(T')$ such that $\deg_{T'}(x)=3$. Suppose there is a path $x,a,b,c$ and a path $x,e$, where $\deg_{T'}(a)=\deg_{T'}(b)=2$, and $\deg_{T'}(c)=\deg_{T'}(e)=1$. If $\{y\} = N(x) - \{a,e\}$, and $T$ is the tree obtained from $T'$ by removing the vertices $x$ and $e$ and adding the edge $ay$, then $T$ is in $\mT$. In addition, $T'$ is obtainable from $T$ by Operation $4$ at $ay$.
\end{lem}

\begin{proof} Assume $\mop(T')=k$. Note that $\{c,e\} \subseteq U(T')$, which implies that neither $a$ nor $y$ is in $U(T')$.  If $x \in U(T')$, then $(U(T') - \{x\}) \cup \{b\}$ is an open packing of $T'$ of size $k$, which contradicts the uniqueness of $U(T')$. Therefore $b\in U(T')$ and $x\notin U(T')$. Note that $U(T') - \{e\}$ is an open packing of $T$, hence $\mop(T) \ge k-1$.  Suppose there exists a $\mop(T)$-set $M$ of size $k$. Then $M$ is a $\mop(T')$-set which is different from $U(T')$. Thus, we infer $\mop(T) = k-1$.

Let $S$ be a $\mop(T)$-set and let $H = T - \{a, b, c \}$. Suppose $|S \cap V(H)|\ge k-2$. Then $(S \cap V(H))\cup\{b,c\}$ is a maximum open packing of $T'$ different from $U(T')$, a contradiction. Therefore, $|S \cap V(H)| = k-3$.
If $y\in S$, then $|S\cap\{a,b,c\}|\le 1$, which implies $|S|\le k-2$, a contradiction. Therefore, $y\notin S$, which implies that $b\in S$. We claim that $y$ is totally dominated by vertices in $S\cap V(H)$. If this claim were not true, then $(S \cap V(H))\cup\{e,x,c\}$ and $(S \cap V(H))\cup\{e,b,c\}$ are two different $\mop(T')$-sets. 
This contradiction says that $y$ is totally dominated by vertices in $S\cap V(H)$, which in turn implies $a\notin S$. Also, $S\cap V(H)=U(T')\cap V(H)$, because $|S\cap V(H)|=|U(T')\cap V(H)|=k-3$, none of the sets contains $y$, and $U(T')$ is unique.
Since $S$ is a maximum open packing of $T$ we conclude that $\{b,c\}\subset S$, which makes $S$ the unique $\mop(T)$-set of size $k-1$. Thus, $T \in \mT$, and $\{a,y\}\cap U(T)=\emptyset$. Hence, $T'$ can be obtained from $T$ by Operation $4$ at $ay$.
\end{proof}

Finally, we are in the case when $\deg_{T'}(x)=3$, there  are two paths $x,a,b,c$ and $x,d,e$ with $\deg_{T'}(a)=\deg_{T'}(b)=\deg_{T'}(d)=2$  and with $c$ and $e$ both leaves of $T'$.

\begin{lem}\label{lem:claim32} Let $T' \in \mT$ and let $x$ be a leaf in $L(T')$ such that $N_{T'}(x)=\{a,d,y\}$. Suppose  there is a path $x,a,b,c$ and a path $x,d,e$, where $\deg_{T'}(a)=\deg_{T'}(b)=\deg_{T'}(d)=2$, and $\deg_{T'}(c)=\deg_{T'}(e)=1$.  If $T_1 = T' - \{a, b, c\}$, then $T_1$ is in $\mT$. In addition, $T'$ is obtainable  from $T_1$ by Operation $2$ at $x$.
\end{lem}

\begin{proof}
Let $\mop(T') = k$.  Let $T_2=T'-\{a,b,c,d,e,x\}$, and let $W=U(T')\cap V(T_2)$. By Lemma~\ref{lem:leaf},
we know $\{c, e\} \subset U(T')$. Thus, neither $a$ nor $x$ is in $U(T')$.  It follows that $b \in U(T')$.
Furthermore, $d \in U(T')$, for otherwise $y \in U(T')$ and then  $(U(T')  -  \{y\}) \cup \{d\}$ is  another $\mop(T')$-set, which is a contradiction.
Hence, $\{b,c,d,e\} \subseteq U(T')$, $x \notin U(T')$,  $y \notin U(T')$ and $W$ is the only open packing of $T_2$ of cardinality $k-4$ that does not contain $y$.
Since $U(T') - \{b,c\}$ is an open packing of $T_1$, we know that $\mop(T_1) \ge k-2$. Suppose there exists an open packing $M$ of $T_1$ of cardinality at least $k-1$. If $x \not\in M$, then $M \cup \{b,c\}$ is an open packing of $T'$ of cardinality at least $k+1$ which is a contradiction.  On the other hand, if $x \in M$, then $(M - \{x\}) \cup \{b, c, e\}$ is an open
packing of $T'$ of cardinality at least $k+1$, which is also a contradiction.  Therefore, no such $M$ exists, and $\mop(T_1) = k-2$.
Let $X = V(T_1) \cap U(T')= U(T') - \{b,c\}$. The set $X$ is a maximum open packing of $T_1$.  If $T_1 \in \mT$, then $X=U(T_1)$ and $T'$ can be obtained from $T_1$ by Operation $2$ at $x$. Hence, it remains to prove that $T_1\in \mT$.

We claim that $X=W \cup \{d,e\}$ is the only open packing of $T_1$ of cardinality $k-2$.  Let $S$ be any open packing of $T_1$ such that $|S|=k-2$.  Suppose that $y \in S$.  Since $S$ is an open packing, $d \notin S$.  Since $S$ is a maximum open packing of $T_1$, we infer that  $|\{x,e\}\cap S|=1$.
Thus $S \cap (V(T_2) -  \{y\})$ is an open packing of $T_2$ of cardinality $k-4$ that does not contain $y$.  From above this means that $S \cap (V(T_2) -  \{y\})=W$.  But then $W \cup \{y,e,b,c\}$ is a $\mop(T')$-set that is different from $U(T')$, which is a contradiction.  Therefore, $y \notin S$.
Since $|S \cap \{x,d,e\}|\le 2$, it follows that $|S \cap V(T_2)| \ge k-4$ and $S \cap V(T_2)$ is an open packing of $T_2$ that does not contain $y$.  Since
$W$ is the only such open packing of $T_2$, we get that $S=W \cup \{d,e\}$, which coincides with $X$.  That is, $T_1\in \mT$, and $T'$ is obtainable from $T_1$ by Operation $2$ at $x$.
\end{proof}

Combining the above lemmas, we infer the following, our main result.
\begin{thm}
Let $T$ be an arbitrary tree. Then $T$ has a unique maximum open packing if and only
$T$ is $P_1$, $P_2$, or can be obtained from $P_2$ by a sequence of Operations 1-4.
\end{thm}

%%%%%%%%%%%%%%%%%%%%%%%%%%%%%%
%%%%%%%%%%%%%%%%%%%%%
%%%%%%%%%%% C O M P U T A T I O N A L    S T U F F   %%%%%%%%%%%%%%%%
%%%%%%%%%%%%%%%%%%%%
%%%%%%%%%%%%%%%%%%%%%%%%%%%%%%

\section{Computational complexity of related problems}
\label{sec:complexity}

Given a graph $G$, consider the subdivision $S(G)$ of $G$. The vertices of $S(G)$ that are the result of the subdivision are {\em subdivided vertices}, while other vertices of $S(G)$ (which correspond to vertices of $G$) are {\em original vertices}. If $e\in E(G)$, then by $v_e$ we denote the subdivided vertex of $S(G)$ that corresponds to the edge $e$. Clearly, $v_e$ is in $S(G)$ adjacent only to two original vertices that correspond to the endvertices of $e$ in $G$.

\begin{thm}
If $G$ is a graph, then
$\mop(S(G))=\alpha(G)+\alpha'(G)$.
\end{thm}
\begin{proof}
First, we prove that $\mop(S(G))\ge\alpha(G)+\alpha'(G)$.
Let $A$ be a maximum independent set and $M$ be a maximum matching in a graph $G$. Denote by $A'$ the set of original vertices in $S(G)$ that correspond to the vertices of $A$, and let $M'$ be the subdivided vertices of $S(G)$ that correspond to the edges of $M$. We claim that $A'\cup M'$ is an open packing of $S(G)$. Note that any two original vertices in $A'$ are at distance at least 4 in $S(G)$. The same holds for any two subdivided vertices in $M'$; indeed, since two edges $e,e'\in M$ have no common endvertex, the distance between the vertices $v_e$ and $v_{e'}$ in $S(G)$ is at least 4. Now, if $a\in A'$ and $v_e\in M'$, then we have two possibilities. If $av_e\in E(S(G))$, then $N(a)\cap N(v_e)=\emptyset$, because the only neighbor of $v_e$ different from $a$ is an original vertex $b$, hence $a$ and $b$ are not adjacent in $S(G)$. But if $a$ and $v_e$ are not adjacent, then their distance is at least 3. By this we conclude that $A'\cup M'$ is indeed an open packing of $S(G)$, thus $\mop(S(G))\ge\alpha(G)+\alpha'(G)$.

For the reversed inequality, consider a maximum open packing $U$ of $S(G)$, and let $A'$ be the set of original vertices in $U$ and $M'$ be the set of subdivided vertices in $U$. (Note that we cannot exclude that one of the sets $A'$ or $M'$ is empty.) Now, let $A$ be the set of original vertices that correspond to $A'$ and $M$ the set of edges that correspond to vertices in $M'$. It is clear that $A$ must be an independent set in $G$, because for two adjacent vertices $x,y\in V(G)$ the corresponding original vertices in $S(G)$ share a common neighbor, which is the vertex $v_{xy}$. Hence, $\alpha(G)\ge |A|=|A'|$. In a similar way, we deduce $\alpha'(G)\ge |M|=|M'|$, because the edges of $M$ must form a matching, since $M'$ is an open packing. Combining the latter inequalities, we get $\alpha(G)+\alpha'(G)\ge |A'|+|M'|=|U|=\mop(S(G))$.
\end{proof}

By a classical result of Edmonds~\cite{edm-1961}, a maximum matching of an arbitrary graph $G$, and hence also $\alpha'(G)$, can be obtained in polynomial time. On the other hand, the decision version of the independence number is a well-known NP-complete problem (its dual invariant vertex cover number appeared in Karp's list of NP-complete problems~\cite{karp}). This implies that the decision version of $\mop$ is an NP-complete problem. Due to our reduction, it is NP-complete even in subdivision graphs (that is, the class of graphs that can be realized as subdivision $S(G)$ of some graph $G$), and consequently also in graphs of girth at least 6. We therefore derive the following result.

\begin{cor}
Let $G$ be a graph, and $k$ a positive integer. The problem of deciding whether $\mop(G)\ge k$ is NP-complete, even if $G$ is restricted to subdivision graphs.
\end{cor}

We next focus on the computational aspects with respect to the classes $\mC$ and $\mT$. To the best of our knowledge, graphs that have a unique maximum open packing have not been studied before. On the other hand, graphs with a unique maximum independent set appeared in several contexts. Recently, Jaume and Molina~\cite{jm-2018} considered the trees with a unique maximum independent set in the context of the so-called null decomposition of trees. This is a decomposition of a tree into two types of trees, one of which are the trees with a unique maximum independent set. In addition, these trees are characterized through a property of their adjacency matrix as follows. The null space ${\cal N}(T)$ of a tree $T$ is defined as the null space of its adjacency matrix, and its {\em support}, $supp({\cal N}(T))$, is the set of coordinates that are non-zero for at least one of the generators of ${\cal N}(T)$. Now, a tree $T$ has a unique maximum independent set if and only if the set $supp({\cal N}(T))$ is a dominating set of $T$. Since the null space of a tree (and therefore also $supp({\cal N}(T))$) can be computed in polynomial time, and it is trivial to check if $supp({\cal N}(T))$ is a dominating set of $T$, we infer that the class of trees with a unique maximum independent set can be recognized efficiently.

Since the independent set and open packing are related problems we wonder whether a similar algebraic approach could work for graphs with unique maximum open packing. In addition, could a more involved analysis of Operations 1-4 from the previous section, result in an efficient algorithm for recognition of the trees in $\mT$? We pose this as a question and suspect it has a positive answer.

\begin{ques}
Is there an algorithm to efficiently recognize the trees in $\mT$?
\end{ques}

%%%%%%%%%%%%%%%%%%%%%%%%
%%%%%%%%%%
%%%%%%%%%%%%%%%%%%%%%%

We now focus on the problem of recognizing the graphs in $\mC$, which we relate to two similar problems. Recall that a {\em 2-packing} in a graph $G$ is a set of vertices with the property that their closed neighborhoods are pairwise disjoint. (The {\em closed neighborhood} of a vertex $x\in V(G)$ is defined as $N[x]=N(x)\cup\{x\}$.) Alternatively, $S$ is a 2-packing in $G$ if the distance between any two distinct vertices in $S$ is at least 3. The {\em 2-packing number}, $\rho(G)$, of $G$ is the maximum cardinality of a 2-packing in $G$.

We denote by $\mCtwo$ the class of graphs $G$ that have a unique maximum 2-packing, and by $\mCalpha$ the class of graphs with a unique maximum independent set. We next prove that the problems of recognizing the graphs in classes $\mC,\mCtwo$ and $\mCalpha$, respectively, are polynomially equivalent problems. In particular, a polynomial algorithm that recognizes graphs from one of the classes implies the existence of a polynomial algorithm to recognize graphs from the other two classes.

We start by showing how the problem of recognizing whether $G$ is a graph with a unique maximum 2-packing translates to the problem of recognizing whether a graph is in $\mC$. For this purpose, we need some more definitions.
The {\em distance} in a graph $G$ between two vertices $x$ and $y$ is defined as the length of a shortest path between $x$ and $y$, and we denote it by $d_G(x,y)$. The {\em diameter} of $G$, ${\rm diam}(G)$, is the maximum distance between vertices in $G$.

Starting from an arbitrary connected graph $G$, we construct a graph $G'$ so that for each edge $e=xy$ in $G$ we add a new vertex $v_e$, and connect $v_e$ by an edge with $x$ and with $y$. Finally, for every two distinct edges $e,f\in E(G)$ we make $v_e$ and $v_f$ adjacent in $G'$.
Note that vertices of the set $C=\{v_e\,:\, e\in E(G)\}$ induce a clique in $G'$,  and for any two vertices $x,y\in V(G)$, we have $d_{G'}(x,y)=\min\{d_G(x,y),3\}$. If $x\in V(G)$ and $e\in E(G)$, then $d_{G'}(x,v_e)\le 2$, and $d_{G'}(x,v_e)=1$ if and only if $x$ is an endvertex of $e$. We also infer that ${\rm diam}(G')=\min\{{\rm diam}(G),3\}$.

\begin{lem}
If $G$ is a connected graph, then $\mop(G')=\rho(G)$.  In addition, $G$ is in $\mCtwo$ if and only if $G'$ is in $\mC$.
\end{lem}
\begin{proof}
Let $G'$ be obtained by the described construction from a (nontrivial) connected graph $G$, and let $S$ be a maximum open packing of $G'$ (the case when $G$ is isomorphic to $K_1$ is clear). Note that any two adjacent vertices $x,y\in V(G)$ have $v_{xy}$ as their common neighbor in $G'$. Hence at most one of the adjacent vertices $x,y$ belongs to $S$.

Suppose that $C\cap S\ne \emptyset$, where $C$ is the set of vertices that are not in $V(G)$ (and arise from edges of $G$). If $v\in C\cap S$, then note that $v$ is the only vertex of $S$. Hence, $\mop(G')=1$, which is only possible when ${\rm diam}(G)\le 2$. But when ${\rm diam}(G)\le 2$, we have $\mop(G')=1=\rho(G)$, $G$ has more than one maximum 2-packing, and also $G'$ has more than one maximum open packing. In this case, the proof is done.

Now, we may assume that $C\cap S=\emptyset$ and ${\rm diam}(G)>2$. Let $T\subset V(G)$. Note that $T$ is a (maximum) 2-packing of $G$ if and only if $T$ is a (maximum) open packing of $G'$. From this we infer $\mop(G')=\rho(G)$, and the second statement
of the lemma also readily follows.
\end{proof}

For the reversed translation we provide the following construction. Let $G$ be an arbitrary graph on $n$ vertices. From $G$ we construct the graph $G^{+}$ as follows. For each vertex $u\in V(G)$ we add six additional vertices denoted $u_1,u_2,u_3,u_1',u_2',u_3'$. Connect with an edge all pairs $u_i,u_j$ when $i\ne j$; also connect each $u_i$ with $u$ and with $u_i'$ for all $i\in\{1,2,3\}$. Finally, for every edge $uv\in E(G)$, add an edge between $u_1'$ and $v$, and add an edge between $v_1'$ and $u$.
Note that $|V(G^{+})|=7n$, and for any $u\in V(G)$ we have $\deg_{G^{+}}(u)=2\deg_G(u)+3$, $\deg_{G^{+}}(u_i)=4$, $\deg_{G^{+}}(u_1')=\deg_G(u)+1$, and $\deg_{G^{+}}(u_2')=1=\deg_{G^{+}}(u_3')$.

\begin{lem}
If $G$ is a graph on $n$ vertices, then $\rho(G^{+})=2n+\mop(G)$.  In addition, $G$ is in $\mC$ if and only if $G^{+}$ is in $\mCtwo$.
\end{lem}
\begin{proof}
Let $S$ be a maximum 2-packing of $G^{+}$. Note that for any $u\in V(G)$, there can be at most three vertices in $S\cap \{u_1,u_2,u_3,u_1',u_2',u_3'\}$, namely,   $u_1',u_2',u_3'$, which form a 2-packing. Also note that $u\notin S$, because $u$ is at distance at most 2 from any other vertex in $\{u_1,u_2,u_3,u_1',u_2',u_3'\}$; from the same reason, $u_i\notin S$ for any $i\in [3]$. Next, it is easy to see that for any $u\in V(G)$, we have $u_2'\in S$ and $u_3'\in S$. Now, if $u$ and $v$ have a common neighbor in $G$, then at most one of the vertices $u_1'$ and $v_1'$ can be in $S$. We infer that the set $\{u\in V(G)\,:\,u_1'\in S\}$ is an open packing of $G$, and it is a maximum open packing of $G$, since $S$ is maximum. Thus, $\rho(G^{+})=2n+\mop(G)$.

In the previous paragraph we established that any maximum 2-packing  of $G^{+}$ is obtained by taking the union of $\{u_2',u_3'\,:\, u\in V(G)\}$ and $\{u_1'\,:\, u\in P\}$, where $P$ is a maximum open packing of $G$. Conversely, note that for any maximum open packing $P$ of $G$, the union of $\{u_2',u_3'\,:\, u\in V(G)\}$ and $\{u_1'\,:\, u\in P\}$ is a maximum 2-packing of $G^{+}$. Both statements imply that $G$ is in $\mC$ if and only if $G^{+}$ is in $\mCtwo$.
\end{proof}

In the next two lemmas, we show that the recognition problems for graphs with a unique maximum independent set and graphs with a unique maximum 2-packing are polynomially equivalent.

Let $G$ be an arbitrary graph, and let $V(G)=\{a_1,\ldots, a_n\}$. From $G$ we construct the graph $G^*$ as follows.  First, we start with the Cartesian product $G\cp K_n$, which is defined as follows. Let vertices of $K_n$ be denoted by $\{v_1,\ldots,v_n\}$.  Then, $V(G\cp K_n)=V(G)\times V(K_n)$, and the edges of $G\cp K_n$ are $(a_i,v_k)(a_j,v_k)$ whenever $a_ia_j\in E(G)$, and $(a_i,v_k)(a_i,v_l)$ for any $a_i\in V(G)$ and any distinct $v_k,v_l\in V(K_n)$.
Now, for each $(i,j)\in [n]\times [n]$ such that $i\neq j$, let $b_{ij}$ and $c_{ij}$ be two new vertices of $G^*$. Make $c_{ij}$ and $b_{ij}$ adjacent in $G^*$, and also connect each $b_{ij}$ with $(a_i,v_j)$ in $G^*$. Finally, for each $i\in [n]$ make $b_{ij}$ and $b_{ik}$ adjacent whenever $j\ne k$.

\begin{lem}
If $G$ is a graph on $n$ vertices, then $\rho(G^*)=n(n-1)+\alpha(G)$. In addition, $G$ is in $\mCalpha$ if and only if $G^*$ is in $\mCtwo$.
\end{lem}
\begin{proof}
Let $C=\{c_{ij}\,:\, (i,j)\in ([n]\times [n]-\{(i,i)\,:\,i\in [n]\})\}$ denote the set of vertices of degree 1 in $G^{*}$, and $B=\{b_{ij}\,:\, (i,j)\in ([n]\times [n]-\{(i,i)\,:\,i\in [n]\})\}$ the set of its neighbors.

Suppose $S$ is a maximum independent set in $G$. Note that then the set $C\cup \{(a_i,v_i)\,:\, a_i\in S\}$ is a 2-packing in $G^{*}$ with size $n(n-1)+\alpha(G)$, which implies $\rho(G^*)\ge n(n-1)+\alpha(G)$. For the reversed inequality first observe that all vertices of $C$ must belong to a maximum 2-packing $P$ of $G^{*}$. Indeed, $P$ should not contain any vertex of $B$, and $(a_i,v_j)\in P$ only if $i=j$. However, if $(a_i,v_i)\in P$, then for any $a_j$, which is adjacent to $a_i$ in $G$, we have $(a_j,v_j)\notin P$. In other words, $\{a_i\,:\, P\cap \{(a_i,v_i)\,:\, i\in [n]\}\neq\emptyset \}$ is an independent set in $G$, which implies that $\rho(G^*)\le n(n-1)+\alpha(G)$.

Since $C$ is in any maximum 2-packing $P$ of $G^*$, this yields the structure of maximum 2-packings of $G^*$, which are formed as the union of $C$ and the set of vertices $(a_i,v_i)$, where vertices $a_i$ form a maximum independent set $I$ of $G$. Therefore, we get that $P$ is a unique maximum 2-packing of $G^*$ if and only if $I$ is a unique maximum independent set of $G$.
\end{proof}

Given a graph $G$, the {\em square} $G^2$ of $G$ is obtained from $G$ by adding edges between any two vertices of $G$ that are at distance $2$.

\begin{lem}
If $G$ is a graph, then $\alpha(G^2)=\rho(G)$. In addition, $G$ is in $\mCtwo$ if and only if $G^2$ is in $\mCalpha$.
\end{lem}
\begin{proof}
Let $S\subset V(G)$ be any set of vertices in $G$. Note that $S$ is a 2-packing in $G$ if and only if $S$ is an independent set in $G^2$. From this we readily derive both statements of the lemma.
\end{proof}

From the above lemmas and translations one infers that the computational complexity for recognizing one of the three classes of graphs implies polynomially equivalent computational complexity for recognizing the other two classes.

\begin{thm}
\label{thm:equivalence}
The recognition of the graphs from classes $\mC,\mCtwo$ and $\mCalpha$ are polynomially equivalent problems.
\end{thm}

We conclude by observing that the recognition of graphs in each of these classes of graphs is likely to be computationally hard.

\begin{thm}
The problem of recognizing the graphs in $\mC$, respectively $\mCtwo$ and $\mCalpha$, is not  polynomial, unless P=NP.
\end{thm}
\begin{proof}
By Theorem~\ref{thm:equivalence}, it suffices to show the statement of the theorem only for the class $\mCalpha$. Suppose that there exists a polynomial time algorithm to decide whether an arbitrary graph $G$ is in the class $\mCalpha$ or not. For a positive integer $r$, let $G_r$ be the graph $G \vee\overline{K_r}$, obtained as the join of $G$ and the edgeless graph on $r$ vertices. Note that the independence number $\alpha(G)$ of $G$ equals the largest integer $r$ such that $G_r$ is not in $\mCalpha$. Hence, checking whether $G_r\in \mCalpha$ for positive integers $r$, starting with $r=|V(G)|$, and decreasing $r$ by 1, until $r$ reaches $\alpha(G)$, results in a polynomial time algorithm to determine $\alpha(G)$. Since the decision version of the independence number of a graph is an NP-complete problem, this is only possible if P=NP.
\end{proof}

\end{document}